\documentclass[11pt,letterpaper]{amsart}
\usepackage[margin=1in]{geometry}

\usepackage{amsmath,amssymb,amsthm,mathrsfs,bbm,comment,units,enumitem,xcolor,adforn,pgfornament,fancyhdr,lastpage,url}
\usepackage[colorlinks, linktocpage, breaklinks]{hyperref}

\usepackage{tikz-cd}
\usetikzlibrary{arrows}

\renewcommand{\leq}{\leqslant}

\newcommand{\piBaseS}[1]{(\href{https://topology.pi-base.org/spaces/S#1}{S{#1}})}
\newcommand{\piBaseP}[1]{(\href{https://topology.pi-base.org/properties/P#1}{P{#1}})}

\newtheoremstyle{mythm}
{.5\baselineskip}	
{.5\baselineskip}	
{}		
{}		
{\bf}	
{. }		
{ }		
{}		

\theoremstyle{mythm}
\newtheorem{theorem}{Theorem}[section]	
\newtheorem{lemma}[theorem]{Lemma}
\newtheorem{proposition}[theorem]{Proposition}
\newtheorem{corollary}[theorem]{Corollary}
\newtheorem{definition}[theorem]{Definition}
\newtheorem{example}[theorem]{Example}

\newtheorem*{remark}{Remark}
\newtheorem{question}{Question}

\newcommand{\term}{\textbf}
\newcommand{\tuple}[1]{\left\langle{#1}\right\rangle}

\title{Separation Axioms Among US}

\author{Steven Clontz}
\address{College of Arts and Sciences, University of South Alabama, 411 N University Blvd North, Mobile, AL 36688, USA}
\email{sclontz@southalabama.edu}

\author{Marshall Williams}
\address{Unafilliated}

\date{\today}

\subjclass{}
\keywords{compactly generated, weak Hausdorff, separation axioms}

\begin{document}

\maketitle

\begin{abstract}
   A standard introductory result is that Hausdorff spaces have the 
   property US, that is,
   each convergent sequence has a unique limit. This paper explores
   several existing and new characterizations of separation axioms
   that are strictly weaker than $T_2$ but strictly stronger than US.
\end{abstract}

\section{Introduction}

The following spectrum of separation
axioms between \(T_1\) and \(T_2\)
was recently surveyed in \cite{2312.08328}.

\[T_2\Rightarrow k_1\text{H}\Rightarrow \text{KC}\Rightarrow
\text{wH}\Rightarrow k_2\text{H}\Rightarrow \text{US}\Rightarrow T_1.\]

None of these implications reverses. For example,
a copy of \(\omega_1+1\) with a doubled endpoint 
\piBaseS{37}\footnote{ID as assigned in the \(\pi\)-Base 
community database of topological counterexamples \cite{PiBase}.}
is an example of a US space (since all non-trivial converging
sequences lay within Hausdorff \(\omega_1\)) which turns
out to not be \(k_2H\).

During a discussion of these properties at the Carolinas
topology seminar, Alan Dow observed that while S37 satisfies
US, it does not have unique limits for all \emph{transfinite} sequences.
To this end, the first author asked a highly-upvoted\footnote{
Albeit, this score was accrued mainly
due its farcical reference of the ``Among Us''/``SUS'' internet memes.
} question \cite{mathse:4778063} on the Mathematics StackExchange forum, seeking an example
of such a ``SUS'' or ``strongly US''
space which is not \(k_2\)H. Such an example was eventually provided by
the second author, which led to the broader work explored in this 
manuscript.


\section{Preliminaries}

The reader is referred to $\pi$-Base \cite{PiBase}
for any definitions not given here.
Consider the following weakenings of $T_2$, which strictly imply $T_1$.

\begin{definition}
    A space is \term{$k_1$-Hausdorff}  \piBaseP{170}
    ($k_1$H for short) provided
    every compact subspace is Hausdorff.
\end{definition}

\begin{definition}
    A space is \term{$k_2$-Hausdorff} \piBaseP{171}
    ($k_2$H for short) provided
    for every compact Hausdorff space $K$ and
    continuous map $f:K\to X$ and points $k,l\in K$ with $f(k)\not=f(l)$,
    there exist open neighborhoods $U,V$ of $k,l$ with $f[U],f[V]$
    disjoint.
\end{definition}

\begin{definition}
    A space is \term{US} \piBaseP{99}
    (``Unique Sequential convergence'') provided
    every convergent sequence has exactly one limit.
\end{definition}

Our investigation will center on the following properties.

\begin{definition}
    A \term{sequence} is a function from $\omega$ into a 
    topological space, while a
    \term{transfinite sequence} is a function
    from some limit ordinal into a topological space. Note that
    while every ($\omega$-length) sequence is
    \term{continuous}, this is not be the case for
    transfinite sequences of longer length.
    
    A \term{limit}
    of a transfinite sequence is a point of the space for
    which every neighborhood contains a final subsequence, that is,
    there exists some point of the sequence for which all
    subsequent points belong to the neighborhood.
    Such transfinite sequences are said to \term{converge}.
\end{definition}

\begin{definition}
    A space is \term{UR} (``Unique Radial convergence'') provided
    any limit of a transfinite sequence is unique.
\end{definition}

\begin{definition}
    A locally compact \piBaseP{130} space $X$ is \term{UOK}
    (``Unique One-point Kompactification'')
    if for every non-compact locally compact Hausdorff $L$ and
    continuous function $f:L\to X$, there exists at most one
    continuous extension $\hat f:\hat L\to X$ where $\hat L=L\cup\{\infty\}$
    is the one-point compactification of $L$. (We call this
    $\hat f(\infty)$ the \term{compactifying limit} of $f$.)
\end{definition}

\begin{definition}
    A space is \term{UCR} (``Unique C-Radial convergence'', 
    called ``SUS'' at \cite{mathse:4778063}) provided
    any limit of a \emph{continuous} transfinite sequence is unique.
\end{definition}





Justification for the choice of some of this terminology is
owed to the reader.

\begin{definition}
A space is \term{Frech\'et-Urysohn} \piBaseP{80} provided
for each $A\subseteq X$ and each $p\in \overline A$ there is a 
sequence of points of $A$ converging to $p$.
\end{definition}

\begin{definition}
A space is \term{sequential} \piBaseP{79} provided
every sequentially closed set is closed, where a set
$A\subseteq X$ is 
\term{sequentially closed} if it contains the limits of all
convergent sequences consisting of points of $A$.  In other words, 
for every non-closed set $A\subseteq X$ there is a point 
$p\in \overline A\setminus A$ and a sequence of points 
of $A$ that converges to $p$.
\end{definition}

\begin{definition}
A space is \term{C-radial} provided
for each $A\subseteq X$ and each $p\in \overline A$ there is a 
\emph{continuous} transfinite sequence of points of $A$ converging to $p$.
\end{definition}

\begin{definition}
A space is \term{pseudo-C-radial} provided
every C-radially closed set is closed,
where a set $A\subseteq X$ is 
\term{C-radially closed} if it contains the limits of all convergent 
\emph{continuous} transfinite sequences consisting of points of $A$.  In other words, 
for every non-closed set $A\subseteq X$ there is a point 
$p\in \overline A\setminus A$ and a 
\emph{continuous} transfinite sequence of points 
of $A$ that converges to $p$.
\end{definition}

\begin{definition}
A space is \term{radial} \piBaseP{172} provided
for each $A\subseteq X$ and each $p\in \overline A$ there is a 
transfinite sequence of points of $A$ converging to $p$.
\end{definition}

\begin{definition}
A space is \term{pseudoradial} \piBaseP{173} provided
every radially closed set is closed, where a set $A\subseteq X$ is 
\term{radially closed} if it contains the limits of all convergent 
transfinite sequences consisting of points of $A$.  In other words, 
for every non-closed set $A\subseteq X$ there is a point 
$p\in \overline A\setminus A$ and a transfinite sequence of points 
of $A$ that converges to $p$.
\end{definition}

The following diagram follows immediately.

\[\begin{tikzcd}
	{\text{Frech\'et-Urysohn}} & {\text{sequential}} \\
	{\text{C-radial}} & {\text{pseudo-C-radial}} \\
	{\text{radial}} & {\text{pseudoradial}}
	\arrow[from=1-1, to=1-2, Rightarrow]
	\arrow[from=1-1, to=2-1, Rightarrow]
	\arrow[from=1-2, to=2-2, Rightarrow]
	\arrow[from=2-1, to=2-2, Rightarrow]
	\arrow[from=2-1, to=3-1, Rightarrow]
	\arrow[from=2-2, to=3-2, Rightarrow]
	\arrow[from=3-1, to=3-2, Rightarrow]
\end{tikzcd}\]

The authors of \cite{zbMATH06427214} use the term
``strongly pseudoradial'' rather than
pseudo-C-radial, and the term
``unique strongly pseudoradial convergence'' rather than
unique C-radial convergence. Our alternative
terminology is based upon the standard term radially closed
used to characterize pseudoradial in \cite{bella2003pseudoradial},
and also the radial property itself,
neither of which not considered in \cite{zbMATH06427214}. It's unclear, if
one chooses the term ``unique strongly pseudoradial convergence'',
what ``unique strongly \emph{radial} convergence'' should mean. Additionally,
it's immediate that every radially closed set is C-radially closed,
but not the converse;
it would be awkward to choose terminology that would require that
``strongly radially closed'' not imply radially closed.

So we restate some results of \cite{zbMATH06427214} using our refactored
language.

\begin{proposition}[Theorem 3.5 of \cite{zbMATH06427214}]\label{thm35of}
    A space $X$ is pseudo-C-radial if and only if for every
    for every non-closed set $A\subseteq X$ there is a point 
    $p\in \overline A\setminus A$ and an \emph{injective} and
    continuous transfinite sequence of points 
    of $A$ that converges to $p$, whose domain is a regular cardinal.
\end{proposition}

\begin{proposition}[Corollary 4.2 of \cite{zbMATH06427214}]
    Suppose $X$ is pseudo-C-radial. Then $X$ is UCR if and only if 
    it is \term{KC} (Kompacts are Closed) \piBaseP{100}.
\end{proposition}



Finally, we will prove a result akin to Proposition \ref{thm35of} for 
UR and UCR.

\begin{theorem}\label{ur-inj}
    Let $X$ be $T_1$.
    Then $X$ is U(C)R if and only if
    any limit of a (continuous and) injective transfinite sequence,
    whose domain is a regular cardinal, is unique.
\end{theorem}
\begin{proof}
    The forwards direction is immediate, so assume the latter and
    take an arbitrary (continuous) transfinite sequence $f:\alpha\to X$
    converging to some limit.
    Let $f^\leftarrow(x)=\{\beta<\alpha:f(\beta)=x\}$.

    Suppose we have some $x_0\in X$ with $f^\leftarrow(x)$ unbounded in $\alpha$.
    Then by $T_1$, $X\setminus\{x_0\}$ is an open set that misses a subsequence
    of $f$, so $x_0$ is the unique limit of $f$.

    Otherwise, $\{f^{\leftarrow}(x):x\in X\}$ is a partition of $\alpha$
    into bounded subsets of $\alpha$. By Lemmas 3.1 and 3.4 of
    \cite{zbMATH06427214} there is a closed unbounded subset $C\subseteq\alpha$
    with $f\upharpoonright C$ injective and $|C|$ regular. Note that, since it
    is closed and unbounded, $C$ is isomorphic to the ordinal space $|C|$. So
    we may consider $f\upharpoonright C$ as a (continuous and) injective
    transfinite sequence whose domain is a regular cardinal, and thus
    its limit $x_0$ is unique.

    Finally, let $x\in X\setminus\{x_0\}$, and since $x$ is not a limit of
    $f\upharpoonright C$ we may choose some $\gamma\in C$ with
    $x\not\in cl(f[C\setminus\gamma])$. Then
    $X\setminus cl(f[C\setminus\gamma])$ is a neighborhood of $x$
    missing a subsequence of $f$, showing $x$ is not a limit of $f$,
    and therefore $x_0$ is the unique limit of $f$.
\end{proof}

\section{Connecting $k_2H$ to US}

We now proceed to establish that UOK and UCR are properties intermediate
to $k_2$H and US.

\begin{theorem}
$k_2$H $\Rightarrow$ UOK $\Rightarrow$ UCR $\Rightarrow$ US.
\end{theorem}

\begin{proof}
Given a $k_2$H space, take non-compact locally compact Hausdorff $L$,
$f:L\to X$ continuous, and points $x_0,x_1$ such that extending $f$
to $\hat f:\hat L\to X$ by either $\hat f(\infty)=x_0$
or $\hat f(\infty)=x_1$ is continuous.
Define $g:K\times\{0,1\}\to X$ by $g(l,i)=f(l)$ and $g(\infty,i)=x_i$. Then $g$ is continuous, and there are no neighborhoods of
$\tuple{\infty,0},\tuple{\infty,1}$ whose images are disjoint, so $g(\infty,0)=g(\infty,1)$, that is, $x_0=x_1$ by $k_2$H.

To see UOK $\Rightarrow$ UCR, note that every limit ordinal is
non-compact locally compact Hausdorff, so given a transfinite sequence,
apply UK to obtain the unique
compactifying limit, which is the unique radial limit.

Finally, UCR $\Rightarrow$ US is immediate as every sequence is
continuous (since its domain $\omega$ is discrete).
\end{proof}

Are these distinct? We defer a witness that $\text{UOK} \not\Rightarrow k_2\text{H}$
to the next section (Example \ref{URnotk2H}),
but let's now demonstrate counterexamples to the latter
two implications.

\begin{example}
    As noted earlier, $\omega_1+1$ with its endpoint doubled \piBaseS{37}
    is an example of a US space that is not UCR: each of the endpoints
    is a limit of an $\omega_1$-length continuous transfinite sequence.
\end{example}

We require a little housekeeping to witness UCR $\not\Rightarrow$ UOK.

\begin{definition}
    A space is \term{sequentially discrete} \piBaseP{167} provided
    that every convergent sequence in the space is eventually constant.
\end{definition}

\begin{lemma}
    Every sequentially discrete space is ``C-radially discrete'':
    every convergent continuous transfinite sequence is eventually constant.
\end{lemma}
\begin{proof}
    Consider a convergent continuous transfinite sequence $f$; suppose it
    is not eventually constant. Every sequentially discrete space is
    $T_1$, so $f$ has infinite range. So we may obtain an $\omega$-length
    injective subsequence $f(\alpha_n)$ for $n<\omega$. But by continuity,
    $f(\alpha_n)\to f(\sup\alpha_n)$; therefore the space is not sequentially
    discrete.
\end{proof}

\begin{corollary}
    Every sequentially discrete space is UCR.
\end{corollary}
\begin{proof}
    All convergent continuous transfinite sequences are eventually constant,
    and eventually constant transfinite sequences have unique limits in $T_1$
    spaces.
\end{proof}

\begin{lemma}
    Let $x$ be a non-isolated point of a $T_3$ space.
    Then $X\setminus\{x\}$ is not compact.
\end{lemma}
\begin{proof}
    Let $\mathcal B$ be a local base at $x$. By regularity, we may
    assume each member of $\mathcal B$ is a closed neighborhood.

    Consider the collection of open sets
    $\mathcal U=\{X\setminus B:B\in\mathcal B\}$.
    Since $X$ is $T_1$, for each $y\in X\setminus\{x\}$ there is some
    $B_y\in\mathcal B$ with $y\not\in B_y$. Thus $\mathcal U$ is
    an open cover of $X\setminus\{x\}$.

    Suppose $\mathcal U$ had a finite subcover
    $\mathcal F=\{X\setminus B:B\in\mathcal B_F\}$ of $X\setminus\{x\}$ for
    $\mathcal B_F\subseteq\mathcal B$ finite. Then
    $\bigcap \mathcal B_F=\{x\}$ would be a neighborhood of $x$, showing
    it is an isolated point.
\end{proof}

\begin{example}
    Let $X$ be a compact, Hausdorff, sequentially discrete space
    with a non-isolated point $x$ (e.g.
    the Stone-Cech compactification $\beta\omega$
    of a countable discrete space \piBaseS{108}).
    
    Let $X'=X\cup\{x'\}$ where $x'$ doubles $x$. Then $X'$ is
    still sequentially discrete, and therefore UCR.
    Note that $L=X\setminus\{x\}$ is non-compact,
    locally compact, and Hausdorff.
    Thus, $L$ with inclusion $\iota:L\to X'$ witnesses
    that $X'$ is not UOK: both $x,x'$ are compactifying points.
\end{example}

\section{Unique Radial convergence}

The UR property sits orthogonal from many of the strengthenings of US
discussed in the introduction, but is still a consequence of $T_2$.

\begin{definition}
    A point $x$ of a space is \term{radially accessible}
    provided there is a non-trivial transfinite sequence that
    converges to it.
\end{definition}

\begin{lemma}
    Every non-isolated point in a Hausdorff compact space is
    radially accessible.
\end{lemma}
\begin{proof}
    Given a non-isolated point $x$ in a Hausdorff compact space $K$,
    let $K_0'=K$. If $K_\alpha'\subseteq K$ is defined for some ordinal
    $\alpha$ such that $x$ is not isolated in the subspace, let
    $K_{\alpha+1}'$ be a compact neighborhood of $x$ in $K_\alpha'$
    such that $K_{\alpha+1}'\subsetneq K_\alpha'$, and note that
    $x$ is not isolated in $K_{\alpha+1}'$. Then if $K_\beta'$ is defined
    for all $\beta<\alpha$ where $\alpha$ is a limit ordinal, then
    let $K_\alpha'=\bigcap_{\beta<\alpha}K_\beta'$.

    Then we have some limit ordinal $\alpha$ such that $\{K_\beta':\beta<\alpha\}$
    is a strictly decreasing chain of compact sets with $x$ non-isolated
    in each. Since $x$ is isolated in $K_\alpha'$, choose a compact neighborhood
    $N$ of $x$ in $K$ such that $K_\alpha'\cap N=\{x\}$, then let
    $K_\beta=K_\beta'\cap N$ for $\beta\leq\alpha$. It follows that
    $x$ is not isolated in each $K_\beta$ for $\beta<\alpha$, and
    $K_\alpha=\{x\}$.

    Then we note that for each neighborhood $U$ of $x$, $U$ contains some
    $K_\beta$; if not then $\{K_\beta\setminus U:\beta<\alpha\}$ has the
    finite intersection property, but its intersection $\{x\}\setminus U$
    is empty, a contradiction of the compactness of $K$. Thus we may pick
    $x_\beta\in K_\beta$ to obtain a transfinite sequence converging to $x$.
\end{proof}

\begin{theorem}
    $T_2 \Rightarrow$ UR $\Rightarrow$ UOK.
\end{theorem}
\begin{proof}
    Given a $T_2$ space,
    suppose $x$ is the limit of a transfinite sequence $f$, and $y\not=x$.
    Let $U,V$ be disjoint open neighborhoods of $x,y$: $U$ contains a final
    subsequence of $f$, so $V$ cannot, and thus $x$ is the unique limit of $f$,
    proving UR.

    Let $X$ be UR, and let $f:L\to X$ be continuous for some non-compact
    locally compact Hausdorff $L$. Let $\hat f_0,\hat f_1:\hat L\to X$
    be continuous extensions of $f$. Note $\infty\in\hat L$ is radially
    accessible, so pick $l_\beta\to\infty$ for $\beta<\alpha$. It follows
    that $\hat f_0(l_\beta)=\hat f_1(l_\beta)$ is a transfinite sequence 
    in UR $X$, so there it has a unique limit which by continuity must
    equal $\hat f_0(\infty)=\hat f_1(\infty)$, showing $X$ is UOK.
\end{proof}

As a result, we now have the following spectrum of properties
among US:

\begin{center}
\begin{tikzcd}
T_2 \arrow[d, Rightarrow] \arrow[rrrr, Rightarrow] &                                 &                                 &                                   & \text{UR} \arrow[d, Rightarrow] &                                  &                                 &     \\
k_1\text{H} \arrow[r, Rightarrow]                  & \text{KC} \arrow[r, Rightarrow] & \text{wH} \arrow[r, Rightarrow] & k_2\text{H} \arrow[r, Rightarrow] & \text{UOK} \arrow[r, Rightarrow] & \text{UCR} \arrow[r, Rightarrow] & \text{US} \arrow[r, Rightarrow] & T_1
\end{tikzcd}
\end{center}

Examples of spaces that are UR but not $k_2$H, and $k_1$H but
not UR, are owed to the reader.

\begin{example}
    The cocountable topology on an uncountable set, e.g. \piBaseS{17},
    is \term{anticompact} \piBaseP{136},
    that is, compact sets must be finite.
    This follows as every infinite set contains a countably infinite
    set, which is closed and discrete. Every anticompact $T_1$ space
    is $k_1$H, since every finite $T_1$ space is $T_2$.
    But every non-trivial transfinite sequence of uncountable length
    converges to every point of the space, and thus the space is not
    UR.
\end{example}

\begin{example}\label{URnotk2H}
    Let $Y$ be a compact and Hausdorff space that is not
    pseudoradial (e.g., again,
    the Stone-Cech compactification $\beta\omega$
    of a countable discrete space \piBaseS{108}).
    Let $A$ be a radially closed set which is not closed
    and $B=Y\setminus A$
    (in $\beta\omega$, $A=\omega$ and $B=\beta\omega\setminus\omega$),
    and let $X=Y\cup B'$, where $B'$ doubles each point of $B$ (so
    $A\cup B$ and $A\cup B'$ are open copies of $Y$ as subspaces of
    $X$).
    Call each pair of corresponding points $b,b'$ from $B,B'$ ``twins'',
    and let $\theta:X\to Y$ send each point $b'\in B'$ to its twin
    $b\in B$, and be the identity elsewhere.

    We claim $X$ is UR (and thus UOK). If $\theta(x_0)\not=\theta(x_1)$, then
    the disjoint neighborhoods of $\theta(x_0),\theta(x_1)$ in Hausdorff
    $Y$ witness disjoint neighborhoods of $x_0,x_1$ in $X$, showing
    no transfinite sequence has both as its limit. If
    $\theta(x_0)=\theta(x_1)$ for $x_0\not=x_1$,
    assume $x_0\in B,x_1\in B'$. Suppose $f:\alpha\to X$ converges
    to $x_0$. Then as $A$ is radially closed, a final tail of $f$
    must lay within $B$ in order for $f\to x_0$. It follows that
    $X\setminus B$ is a neighborhood of $x_1$ that does not contain
    a tail of $f$, so $f\not\to x_1$. Similarly, no transfinite sequence
    converging to $x_1$ may converge to $x_0$.

    Finally, $X$ may be seen to fail $k_2$H by considering the map
    $h:Y\times\{0,1\}\to X$ sending $h(a,i)=a$ for $a\in A$,
    and $h(b,0)=b\in B,h(b,1)=b'\in B'$ for $b\in B$. Since $A$
    is not closed, choose $l\in\operatorname{cl}(A)\setminus A$.
    Then $\tuple{l,0},\tuple{l,1}$ are points of $Y\times\{0,1\}$
    with $h(l,0)=l\not=l'=h(l,1)$, but for every pair of neighborhoods
    $U,U'$ of $\tuple{l,0},\tuple{l,1}$,
    we have $h[U]\cap h[U']\cap A\not=\emptyset$.
\end{example}

\section{Other weakenings of $T_2$}

Combined with the results of \cite{2312.08328}, we have the
following.

\begin{definition}
    A space is \term{locally Hausdorff} (lH for short, \piBaseP{84}) 
    provided each point has a Hausdorff neighborhood.
\end{definition}

\begin{definition}
    A space is \term{semi-Hausdorff} (sH for short, \piBaseP{169})
    provided for each pair of distinct points $x,y$, there exists
    a regular open neighborhood of $x$ missing $y$, where a regular
    open neighborhood equals the interior of its closure.
\end{definition}

\begin{definition}
    A space \term{has closed retracts} (RC for short, \piBaseP{101})
    provided every retract subspace ($A\subseteq X$ with a continuous 
    $f:X\to A$ fixing each point of $A$) is closed.
\end{definition}

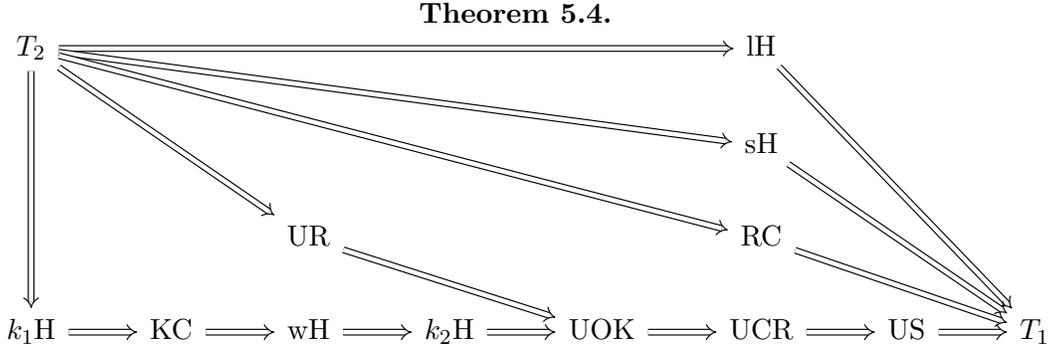
\begin{theorem}\label{all-props}
\begin{center}
\begin{tikzcd}
T_2 \arrow[ddd, Rightarrow] \arrow[rrrrr, Rightarrow] \arrow[rrrrrd, Rightarrow] \arrow[rrrrrdd, Rightarrow] \arrow[rrdd, Rightarrow] &                                 &                                   &                                   &                                 & \text{lH} \arrow[rrddd, Rightarrow] &                                 &     \\
                                                                                                                                      &                                 &                                   &                                   &                                 & \text{sH} \arrow[rrdd, Rightarrow]  &                                 &     \\
                                                                                                                                      &                                 & \text{UR} \arrow[rrd, Rightarrow] &                                   &                                 & \text{RC} \arrow[rrd, Rightarrow]   &                                 &     \\
k_1\text{H} \arrow[r, Rightarrow]                                                                                                     & \text{KC} \arrow[r, Rightarrow] & \text{wH} \arrow[r, Rightarrow]   & k_2\text{H} \arrow[r, Rightarrow] & \text{UOK} \arrow[r, Rightarrow] & \text{UCR} \arrow[r, Rightarrow]    & \text{US} \arrow[r, Rightarrow] & T_1
\end{tikzcd}
\end{center}
\end{theorem}

Examples of spaces which are (respectively) lH, sH, and RC, but not
US (and thus not UR) are already listed in $\pi$-Base. But does UR imply
any of lH, sH, or RC?
Reader, it does not.

\begin{example}
    Example \ref{URnotk2H} was shown to be UR, but it is not RC:
    the map $\theta: X\to Y$ is continuous, but $Y$ is not closed
    in $X$.
\end{example}

\begin{example}
    Consider the set $X=\omega$ with the topology given by
    $\mathcal F\cup\{\emptyset\}$ for some free ultrafilter on $\omega$
    \piBaseS{145}. Since the space is countable,
    UR is equivalent to US; the space is US (in fact, $k_1$H)
    because it is $T_1$ and anticompact.

    However, the space is not sH because it is \term{hyperconnected}
    \piBaseP{39}: every nonempty open set is dense. Therefore
    the only regular open set in $X$ is the entire space.
    Neither is it locally Hausdorff: every open subspace is
    still hyperconnected, and therefore not Hausdroff.
\end{example}

\begin{theorem}
    The figure given as Theorem \ref{all-props} is complete:
    each pair of properties not connected by arrows is witnessed
    by a known counterexample (given either here or in
    \cite{2312.08328}).
\end{theorem}

\section{Spaces generated by maps}

We will conclude this investigation
by observing that there exists a general framework that can be used
to characterize many of these separation axioms.  

The aforementioned framework will be developed in the next section.  In preparation for this,  we review here some facts on topologies generated by maps from certain classes of spaces. 


\begin{definition}
    Let $\mathbf C$ be a class of topological spaces. Given a topological
    space $X$, let $X_{\mathbf C}$ have the (possibly) finer topology
    generated by continuous maps from spaces in $\mathbf C$ into $X$;
    i.e., a set $U$ is open in $X_{\mathbf C}$ whenever $f^\leftarrow[U]$
    is open in $Z$ for all $Z\in\mathbf C$ and continuous $f:Z\to X$.
    We say a subset of $X$ is \term{$\mathbf C$-closed}
    (resp. $\mathbf C$-open) provided it is closed (resp. open)
    in $X_{\mathbf C}$.
\end{definition}
\newcommand{\gen}[1]{#1_{\text{Gen}}}

\begin{definition}
    Let $\mathbf C$ be a class of topological spaces.
    We say a space $X$ is \term{$\mathbf C$-generated} if every
    $\mathbf C$-closed set is closed (equivalently, every
    $\mathbf C$-open set is open), that is, the topology on
    $X_{\mathbf C}$ coincides with its original topology.  We denote by $\gen{\mathbf{C}}$ the class of $C$-generated spaces.
\end{definition}

\begin{remark}
\label{basic-stuff}
    It is easy to check that $\mathbf C\subseteq \gen{\mathbf C}$, and that if $f\colon X\to Y$ is continuous, then the induced map $f_{\mathbf C}\colon X_{\mathbf C}\to Y_{\mathbf C}$ is continuous as well.
\end{remark}

\begin{remark}
    It is easy to verify from the definition that for any nonempty class $\mathbf C$ of topological spaces, $\gen{\mathbf C}$ is closed under homeomorphic equivalence, the formation of topological quotients, and the formation of topological coproducts (i.e., disjoint unions).  It follows from Theorem A of \cite{Kennison} that  $\gen{\mathbf C}$ is co-reflective, as defined there and elsewhere in the literature.

    Moreover, there is idempotence in that $\gen{(\gen{\mathbf C})}=\gen{\mathbf C}$.  From this, it is easy to see that every co-reflective sub-category has the form $\gen{\mathbf C}$ for some nonempty class of spaces $\mathbf C$.

We will not discuss the general theory of co-reflective subcategories here, but refer the interested reader to \cite{Herrlich} \cite{HerrlichStrecker} and \cite{Kennison}, and the references therein. 

\end{remark}

We establish some examples here for the reader's convenience, directing
the reader to \cite{HerrlichStrecker} (see, e.g., the chart on page 212)
where many of these were also considered.

\begin{definition}
    We define the following useful classes of topological spaces.
    \begin{itemize}
    \item $\mathbf{P}$ is the class consisting of a two-point indiscrete space
    $P=\{0,1\}$.
        \item $\mathbf A$ is the class containing only the
        two-point Sierpinski space $A=\{0,1\}$ with topology
        $\{\emptyset,\{0\},\{0,1\}\}$.
        \item $\mathbf{S}$ is the class containing only
        the converging sequence $\omega+1$.
        \item $\mathbf{CR}$ is the class of all successors $\kappa+1$
        of regular cardinals with their order topologies.
        \item $\mathbf R$ is the class of all
        successors $\kappa+1$ of
        regular cardinals, where each point $\alpha<\kappa$ is isolated
        and $\kappa$ itself has its neighborhoods from the order topology.
        \item $\mathbf{K_1}$ is the class of all compact
        spaces.
        \item $\mathbf{K_2}$ is the class of all compact Hausdorff
        spaces.
        \item $\mathbf H$ is the class of all Hausdorff spaces.
    \end{itemize}
\end{definition}

In particular, the notion of \term{compactly-generated} appears
in the literature in both the sense of $\mathbf K_1$-generated
(e.g. \cite{zbMATH02107988}; we will call such examples
\term{$k_1$-spaces} \piBaseP{140}) and $\mathbf K_2$-generated
(e.g. \cite{MR2273730}; we will call such examples
\term{$k_2$-spaces} \piBaseP{141}). And it is known that what we've called
$k_1$-Hausdorff (resp. $k_2$-Hausdorff) is exactly
$\mathbf K_1$-Hausdorff (resp. $\mathbf K_2$-Hausdorff).

\begin{example}
    A space is \term{topologically partitioned} \piBaseP{185}
    (it has a basis that partitions the space)
    if and only if it is $\mathbf{P}$-generated.
\end{example}
\begin{proof}
    Let $X$ be topologically partitioned, and suppose $U$ is $\mathbf P$-open
    with $x\in U$. Let $V$ be a basic open neighborhood of $x$, and
    $y\in V$. It follows that $x,y$ are topologically indistinguishable,
    and the map $f:P\to X$ defined by $f(0)=x,f(1)=y$ is continuous.
    Since $U$ is $\mathbf P$-open, it follows that $f^\leftarrow[U]=\{0,1\}$,
    and thus $y=f(1)\in U$. This shows $V\subseteq U$, and $U$ is open.

    Let $X$ be $\mathbf P$-generated, and for each $x$ let $[x]$ collect
    all points of $X$ topologically indistinguishable from it. It follows that
    $[x]$ is $\mathbf P$-open, and therefore open. Finally, we conclude that
    $\{[x]:x\in X\}$ is a partition of $X$ forming a basis for its topology.
\end{proof}
\begin{example}
    A space is \term{Alexandrov} \piBaseP{90} (the intersection of open sets is open)
    if and only if it is $\mathbf A$-generated.
\end{example}
\begin{proof}
    Let $X$ be Alexandrov and suppose $U$ is $\mathbf A$-open
    with $x\in U$.
    Suppose there is some point $y\in X\setminus U$ which belongs
    to every neighborhood of $x$. Then the function $f:A\to X$
    defined by $f(0)=y,f(1)=x$ is continuous. But then
    $f^\leftarrow[A]=\{1\}$ is not open, a contradiction. So for
    each $y\in X\setminus U$ there is a neighborhood of $x$
    missing $y$, and the intersection of these neighborhoods is
    a neighborhood of $x$ which is a subset of $U$, proving it is open.

    Let $X$ be $\mathbf A$-generated, and let $\mathcal U$ be a
    collection of open sets. Let $f:A\to X$ be continuous and
    consider $f^\leftarrow\left[\bigcap\mathcal U\right]$. For
    each $U\in\mathcal U$ we have $f^\leftarrow[U]$ open; if
    any is empty, then so is $f^\leftarrow\left[\bigcap\mathcal U\right]$.
    Otherwise, they each contain $0$, and so does
    $f^\leftarrow\left[\bigcap\mathcal U\right]$, making
    $f^\leftarrow\left[\bigcap\mathcal U\right]$ open.
    Therefore $\bigcap\mathcal U$ is $\mathbf A$-open and thus open.
\end{proof}

The following follow as sequentially (resp. C-radially, radially)
closed is equivalent to $\mathbf S$- (resp. $\mathbf{CR}$-, $\mathbf R$-) closed.

\begin{example}
    A space is sequential if and only if it is $\mathbf S$-generated.
\end{example}
\begin{example}
    A space is pseudo-C-radial if and only if it is $\mathbf{CR}$-generated.
\end{example}
\begin{example}
    A space is pseudoradial if and only if it is $\mathbf R$-generated.
\end{example}

We finally see that $\mathbf H$ is an extremal case.

\begin{theorem}\label{extremalH}
    Every topological space is $\mathbf H$-generated.
\end{theorem}
\begin{proof}
    Suppose $x\in\overline{H}\setminus H$ for some subset $H\subseteq X$.
    
    If $x\in\overline{\{h\}}$
    for some $h\in H$, consider $f:\omega+1\to X$ with $f(n)=h$ and
    $f(\omega)=x$. This $f$ is continuous, and $f^\leftarrow[H]=\omega$
    is not closed, so it follows that $H$ is not $\mathbf H$-closed.

    Otherwise, let $Z$ be the refinement of the subspace $H\cup\{x\}$ isolating
    each point of $H$. Then $Z$ is Hausdorff, since as $x\not\in\overline{\{h\}}$
    there is some neighborhood of $x$ disjoint from $\{h\}$. Then if
    $f:Z\to X$ is the inclusion map, note $f^\leftarrow[H]=H$ is not closed
    in $H\cup\{x\}$, so it follows that $H$ is not $\mathbf H$-closed.
\end{proof}

\begin{remark}
This proof shows that
we could have refined $\mathbf H$ to the family of all discrete or almost discrete \piBaseP{203} Hausdorff spaces.
\end{remark}

\section{A generalized framework for separation axioms}

With the framework of $\mathbf C$-generated spaces, comes a dual notion of a $\mathbf C$-Hausdorff space.

\begin{definition}
    Let $\mathbf C$ be a class of topological spaces.
    We say a space $X$ is \term{$\mathbf C$-Hausdorff} if its diagonal
    $\Delta=\{\tuple{x,x}:x\in X\}\subseteq X^2$ is
    $\mathbf C$-closed in $X^2$.
\end{definition}

First, we note the contravariance between $\mathbf C$-generated and $\mathbf C$-Haudorff notions.

\begin{theorem}
\label{contravariance1}
    If every member of $\mathbf C$ is $\mathbf D$-generated, then
    then every $\mathbf D$-Hausdorff space is $\mathbf C$-Hausdorff.
\end{theorem}
\begin{proof}
    Let $X$ be $\mathbf D$-Hausdorff, let $Z\in\mathbf C$, and let
    $f\colon Z\to X\times X$ be continuous.  For every $W\in\mathbf D$ and continuous $h:W\to Z$, note by the $\mathbf D$-Hausdorff condition, 
    $h^\leftarrow[f^\leftarrow[\Delta_X]]$ is closed.  It then follows by definition that $f^\leftarrow[\Delta_X]$ is $\mathbf D$-closed.  By assumption, since $Z$, as a member of $\mathbf C$, is $\mathbf D$-generated, we see that $f^\leftarrow[\Delta_X]$ is closed, and since $f$ was arbitrary, $\Delta_X$ is $\mathbf C$-closed.
\end{proof}

\begin{corollary}
    If $\mathbf C\subseteq\mathbf D$, then every $\mathbf C$-generated space
    is $\mathbf D$-generated, and every $\mathbf D$-Hausdorff space is
    $\mathbf C$-Hausdorff.
\end{corollary}
\begin{proof}
    The first statement is immediate from the definitions, and the second follows from Theorem \ref{contravariance1}.
\end{proof}

We close with some concrete examples, dual to the picture outlined in the preceding section.

\begin{theorem}
    A space is $T_0$ if and only if it is $\mathbf{P}$-Hausdorff.
\end{theorem}
\begin{proof}
    Let $X$ be $T_0$ and consider $f:P\to X^2$ continuous.
    We will show that if $f(0)=\tuple{x,x}$ then $f(1)=\tuple{y,y}$; in turn,
    $f^\leftarrow[\Delta]$ is always $\emptyset$ or $P$ (and thus closed).
    If $f(1)=\tuple{y,z}$ with $y\not=z$, assume (without loss of generality)
    that $y\not=x$. Then either we have an open set $U$ with $x\in U,y\not\in U$
    or $x\not\in U,y\in U$. It follows then either $f^\leftarrow[U\times X]=\{0\}$
    or $f^\leftarrow[U\times X]=\{1\}$; either way, $f^\leftarrow[U\times X]$
    is not open, a contradiction.

    If $X$ is not $T_0$, let $x,y$ be distinct but topologically indistinguishable
    points.
    Then $f:P\to X^2$ defined by $f(0)=\tuple{x,x},f(1)=\tuple{x,y}$ is
    continuous, but $f^\leftarrow[\Delta]=\{0\}$ is not closed, showing
    $X$ is not $\mathbf P$-Hausdorff.
\end{proof}
\begin{theorem}
    A space is $T_1$ if and only if it is $\mathbf A$-Hausdorff.
\end{theorem}
\begin{proof}
    Let $X$ be $T_1$ and let $f,g:A\to X$ continuous have $f(z)\not=g(z)$.
    If $z=0$ we have $f[\{0\}]\cap g[\{0\}]=\emptyset$ immediately.
    If $z=1$, pick $U$ open with $f(1)\in U$ but $U\cap g[A]=\emptyset$.
    Then $f^\leftarrow[U]$ is an open neighborhood of $1\in A$, so
    $f^\leftarrow[U]=A$. Then $f[A]=U$ and thus $f[A]\cap g[A]=\emptyset$.

    Let $X$ be not $T_1$, so there exist $x\not=y$ where every neighborhood
    of $x$ contains $y$. Then $f,g:A\to X$ defined by $f(0)=y,f(1)=x$
    and $g(0)=y,g(1)=y$ are continuous with $f(0)\not=g(0)$. Then
    $f[\{0\}]\cap g[\{0\}]=\emptyset$, showing $X$ is not $\mathbf A$-Hausdorff.
\end{proof}

\begin{theorem}
    A space is US (resp. UR) if and only if it is $\mathbf S$-
    (resp. $\mathbf R$-) Hausdorff.
\end{theorem}
\begin{proof}
    Let $X$ be US/UR and $f,g:\kappa+1\to X$. If $f(\alpha)\not=g(\alpha)$
    for $\alpha<\kappa$ then $f[\{\alpha\}]\cap g[\{\alpha\}]=\emptyset$
    and we're done. If $f(\kappa)\not=g(\kappa)$,
    $f,g$ must have $f[U]\cap g[U]=\emptyset$ for some neighborhood of $\kappa$;
    if not, let
    $x_\alpha\in f[\kappa\setminus \alpha]\cap g[\kappa\setminus \alpha]$,
    and note $x_n$ converges to both $f(\kappa)$ and $g(\kappa)$.

    Similarly, if $X$ is not US/UR and $x_\alpha$ converges to both $x,y$, then
    we may define $f,g:\kappa+1\to X$ continuous by
    $f(\alpha)=g(\alpha)=x_\alpha$ and
    $f(\kappa)=x,g(\kappa)=y$, and thus $f[U]\cap g[U]\not=\emptyset$
    for every neighborhood $U$ of $\kappa$.
\end{proof}

\begin{theorem}
    A space is UCR if and only if it is $\mathbf{CR}$-Hausdorff.
\end{theorem}
\begin{proof}
    This also follows from the above proof, but with more care to consider
    each non-isolated limit ordinal rather than only the maximum.
\end{proof}

\begin{theorem}
    A space is $T_2$ if and only if it is $\mathbf H$-Hausdorff.
\end{theorem}
\begin{proof}
    A space is $T_2$ if and only if its diagonal $\Delta$ is closed,
    and by Theorem \ref{extremalH}, closed is equivalent to $\mathbf H$-closed.
\end{proof}

\section{Questions}

\begin{question}
    Is there some class $\mathbf{OK}$ such that
    $\mathbf{OK}$-Hausdorff is equivalent to UOK?
\end{question}

\begin{question}
    While all the properties listed in Theorem
    \ref{all-props} are distinct, what are sufficient conditions for
    some to become equivalent? For example, all compact KC spaces are RC,
    but \piBaseS{192} is a compact RC space that is not even US; however,
    perhaps RC and KC are equivalent among compact US spaces?
\end{question}

\section{Acknowledgments}

The authors dedicate this paper to the memory of Peter Nyikos.

The first author wants to especially honor his memory. Peter was the
first mathematician outside of the author's institution to take notice
of the author's research on characterizing topological properties using
$\omega$-length games.
The author appreciates Peter's warm encouragement of his work,
and credits it as a huge
enabler of any research accomplishments the author might claim from his early
career, particularly the author's very first publication
answering Peter's question at
\cite{nyikos2014proximal}, asking whether all proximal compact spaces are
Corson compact. Peter is missed, but his work will go on to inspire
mathematicians for many years to come.

\bibliographystyle{amsplain}
\bibliography{ref}

\end{document}